\documentclass[reqno]{amsart}

\usepackage{amsmath}
\usepackage{amsfonts}
\usepackage{amssymb,enumerate}
\usepackage{amsthm}
\usepackage[all]{xy}
\usepackage{color}

\theoremstyle{plain}
\newtheorem{lem}{Lemma}[section]
\newtheorem{cor}[lem]{Corollary}
\newtheorem{prop}[lem]{Proposition}
\newtheorem{thm}[lem]{Theorem}
\newtheorem{conj}[lem]{Conjecture}

\newtheorem*{mthm*}{Main Theorem}

\theoremstyle{definition}

\newtheorem{ex}[lem]{Example}

\newtheorem{para}[lem]{}

\newtheorem*{convention*}{Convention}

\newcommand{\id}{\operatorname{id}}

\newcommand{\lotimes}{\otimes^{\mathbf{L}}}
\newcommand{\HH}{\operatorname{H}}
	
\newcommand{\coker}{\operatorname{Coker}}

\newcommand{\AR}{\operatorname{\texttt{(AR1)}}}
\newcommand{\Ar}{\operatorname{\texttt{(AR2)}}}
\newcommand{\im}{\operatorname{Im}}
\newcommand{\shift}{\mathsf{\Sigma}}

\newcommand{\End}{\operatorname{End}}

\newcommand{\Ker}{\operatorname{Ker}}

\newcommand{\p}{\ideal{p}}

\newcommand{\xra}{\xrightarrow}

\renewcommand{\geq}{\geqslant}
\renewcommand{\leq}{\leqslant}
\renewcommand{\ker}{\Ker}

\newcommand{\Ext}[4][R]{\operatorname{Ext}_{#1}^{#2}(#3,#4)}

\newcommand{\Hom}{\operatorname{Hom}}

\def\Ext{\operatorname{Ext}}

\def\p{\mathfrak{p}}

\def\D{\operatorname{\mathsf{D}}}

\def\D{\mathcal{D}}
\def\w{\text{\textsf{w}}}
\def\C{\mathcal{C}}
\def\K{\mathcal{K}}

\def\End{\mathrm{End}}

\numberwithin{equation}{lem}

\begin{document}

\bibliographystyle{amsplain}

\title[Diagonal tensor algebra and na\"{\i}ve liftings]{Diagonal tensor algebra and na\"{\i}ve liftings}

\author{Saeed Nasseh}
\address{Department of Mathematical Sciences\\
Georgia Southern University\\
Statesboro, GA 30460, U.S.A.}
\email{snasseh@georgiasouthern.edu}

\author{Maiko Ono}
\address{Institute for the Advancement of Higher Education, Okayama University of Science, Ridaicho, Kitaku, Okayama 700-0005, Japan}
\email{ono@ous.ac.jp}

\author{Yuji Yoshino}
\address{Graduate School of Environmental, Life, Natural Science and Technology, Okayama University, Okayama 700-8530, Japan}
\email{yoshino@math.okayama-u.ac.jp}

\thanks{Y. Yoshino was supported by JSPS Kakenhi Grant 19K03448.}



\begin{abstract}
The notion of na\"{\i}ve lifting of DG modules was introduced by the authors in~\cite{NOY, NOY1} for the purpose of studying problems in homological commutative algebra that involve self-vanishing of Ext. Our goal in this paper is to deeply study the na\"{\i}ve lifting property using the tensor algebra of the shift of the diagonal ideal (or, diagonal tensor algebra, as is phrased in the title of this paper). Our main result provides several characterizations of na\"{\i}ve liftability of DG modules under certain Ext vanishing conditions. As an application, we affirmatively answer \cite[Question 4.10]{NOY2} under the same assumptions.
\end{abstract}

\maketitle

\tableofcontents

\section{Introduction}\label{sec20200314a}

Lifting theory of differential graded (DG) modules along certain DG algebra homomorphisms was initially studied by Nasseh and Sather-Wagstaff~\cite{nasseh:lql} in order to gain a better understanding on a question of Vasconcelos; see~\cite{Vasc}. Since then, this theory has been significantly developed by the authors in~\cite{NOY, NOY1, NOY3, nassehyoshino, OY} for the purpose of studying the following long-standing conjecture in homological commutative algebra, known as the \emph{Auslander-Reiten Conjecture}.

\begin{conj}[\cite{AR}]\label{conj20230122a}
A finitely generated module $M$ over a commutative noetherian local ring $R$ is free if $\Ext^{i}_R(M,M\oplus R)=0$ for all $i\geq 1$.
\end{conj}

While this conjecture is still open in general, it has been proven affirmatively for several classes of commutative noetherian local rings; see, for instance, \cite{AY, auslander:lawlom, avramov:svcci, avramov:edcrcvct, avramov:phcnr, huneke:vtci, huneke:voeatoscmlr, jorgensen:fpdve, nasseh:vetfp, nasseh:lrqdmi, nasseh:oeire, MR1974627, sega:stfcar}, to name a few.
The following discussion briefly explains a relation between this conjecture and lifting theory of DG modules.

\begin{para}\label{para20230424s}
In Conjecture~\ref{conj20230122a}, the ring $R$ can be assumed to be complete. Let $R\cong Q/I$ be a minimal Cohen presentation, i.e., $Q$ is a regular local ring and $I$ is an ideal of $Q$. Tate~\cite{Tate} showed that there is a (finite or infinite) free extension of DG algebras
\begin{equation}\label{eq20230424c}
Q\hookrightarrow Q'=Q\langle X_n\mid n\leq \infty\rangle
\end{equation}
such that $Q'$ is quasiisomorphic to $R$. Hence, there is an equivalence $\D(Q')\simeq \D(R)$ between the derived categories under which Conjecture~\ref{conj20230122a} can be translated to an equivalent problem for a DG module $N$ over the DG algebra $Q'$ with corresponding Ext vanishing assumptions; see~\ref{para20230521a} for details. Since $Q$ is a regular ring, developing a \emph{suitable} notion of lifting along the DG algebra extension~\eqref{eq20230424c} that lifts $N$ to a complex over $Q$ will result in an affirmative answer to Conjecture~\ref{conj20230122a}. 
\end{para}

The classical notions of lifting for DG modules studied in~\cite{NOY, nasseh:lql, nassehyoshino, OY}, namely, \emph{lifting} and \emph{weak lifting}, are generalizations of those for modules and complexes studied by Auslander, Ding, Solberg~\cite{auslander:lawlom} and Yoshino~\cite{yoshino}. As we mention in the introduction of~\cite{NOY3}, these versions of lifting are not suitable in the sense of~\ref{para20230424s}, as they do not allow working with free extensions of DG algebras that are obtained by adjoining more than one variable, e.g., \eqref{eq20230424c}. To overcome this issue, the authors introduced the notion of \emph{na\"{\i}ve lifting} for DG modules in~\cite{NOY, NOY1} and proved the following result; see~\ref{defn20230127a} for the definition of na\"{\i}ve lifting.

\begin{thm}[\protect{\cite[Main Theorem]{NOY1}}]\label{thm20230127f}
Let $Q$ be a commutative ring (not necessarily regular). Assume that $Q'=Q \langle X_n\mid n<\infty \rangle$ is obtained by adjoining finitely many variables of positive degrees to $Q$.
If $N$ is a bounded below semifree DG $Q'$-module with $\Ext _{Q'} ^i (N, N)=0$ for all $i\geq 1$, then  $N$ is na\"ively liftable to $Q$.
\end{thm}

While, by Theorem~\ref{thm20230127f}, na\"{\i}ve lifting property holds along \emph{finite} free extensions of DG algebras under the above Ext vanishing assumptions, it is still an open problem whether it holds along \emph{infinite} free extensions of DG algebras, even under stronger assumptions on the vanishing of Ext.
However, the following statement holds true.

\begin{thm}[\protect{\cite[Theorem 7.1]{NOY1}}]\label{thm20230428a}
If na\"{\i}ve lifting property holds along the free extension of DG algebras~\eqref{eq20230424c}, then Conjecture~\ref{conj20230122a} holds.
\end{thm}

Theorems~\ref{thm20230127f} and~\ref{thm20230428a} provide another proof for the following well-known fact; see the discussion in~\cite[7.4]{NOY1}.

\begin{cor}\label{cor20230425w}
If $R$ is a complete intersection local ring, then Conjecture~\ref{conj20230122a} holds.
\end{cor}

Theorems~\ref{thm20230127f} and~\ref{thm20230428a} show that na\"{\i}ve lifting is a suitable version of lifting in the sense of~\ref{para20230424s}. Hence, the goal of this paper is to deeply study na\"{\i}ve lifting property of DG modules under certain Ext vanishing assumptions along a homomorphism $A\to B$ of DG algebras, where $B$ is a semifree DG module over $A$, e.g., where $B=A \langle X_i\mid n\in \mathbb{N} \rangle$ is an infinite free extension of DG algebras; see Theorem~\ref{equivalent conditions}. The main tool that we use for our study is the tensor algebra of the shift of the diagonal ideal (or, diagonal tensor algebra, as is phrased in the title of this paper).\vspace{6pt}

The organization of this paper is as follows.\vspace{6pt}

In Section~\ref{sec20221213a}, we specify some of the terminology and provide the background on the diagonal ideal, which we denote by $J$. We also introduce the diagonal tensor algebra, which is denoted by $T$, and explain that in addition to the usual grading as a DG algebra, $T$ has another grading, which we refer to as the tensor grading. 
Then, in Section~\ref{sec20230425d} we introduce the notions of tensor graded DG $T$-modules and tensor graded DG $T$-module homomorphisms. Furthermore, for a semifree DG $B$-module $N$, we construct a map, denoted by $\omega_N$, that is shown in Theorem~\ref{equivalent conditions} to be the obstruction to na\"{\i}ve lifting, and hence, to Conjecture~\ref{conj20230122a}. The map $\omega_N$ is an element of a certain tensor graded endomorphism ring, which we denote by $\Gamma_{N \otimes _B T}$.\vspace{6pt}

In Section~\ref{sec20230425s}, we discuss the action of shifts of $\omega_N$ as an element in the ring $\Gamma_{N \otimes _B T}$ on certain graded right $\Gamma_{N \otimes _B T}$-modules. Our main result in this section is Theorem~\ref{main} whose proof involves several steps, including an analysis of the mapping cone of $\omega_N$.
In Section~\ref{sec20230425r}, we study the relationships among certain endomorphism rings. Our first main result in this section, namely Theorem~\ref{omega generates end}, states that 
$\Gamma _{N \otimes _B T}$ can be described as a polynomial ring in the variable $\omega_N$ over another endomorphism ring under certain assumptions for DG modules. Our second main theorem in this section is Theorem~\ref{kernel of omega} in which we determine the kernel and cokernel of the map that is defined by the left multiplication by $\omega_N$ on $\Gamma _{N \otimes _B T}$. These theorems are crucial in our study and are used in the subsequent section.\vspace{6pt}

Section~\ref{sec20230422a} is where we prove our main result in this paper, namely Theorem~\ref{equivalent conditions}, in which we give several characterizations of na\"{\i}ve lifting property under certain Ext vanishing conditions. The proof of this theorem uses our entire work from the previous sections. As an immediate application of Theorem~\ref{equivalent conditions}, we provide an affirmative answer to~\cite[Question 4.10]{NOY2} under the same assumptions; see Corollary~\ref{cor20230807a}. Finally, we conclude this paper with Appendix~\ref{sec20230521a} in which we briefly discuss some details on the conditions that we expose on the DG modules throughout the paper and their relationships with the assumptions in Conjecture~\ref{conj20230122a}.

\section{Background and setting}\label{sec20221213a}

We assume that the reader is familiar with differential graded (DG) homological algebra; general references on this subject include~\cite{avramov:ifr,avramov:dgha, felix:rht, GL}.
The purpose of this section is to fix our notation, specify some of the terminology, and briefly provide the necessary background on the diagonal ideal and tensor algebra.

\begin{para}
Throughout the paper, $R$ is a commutative ring and $(A,d^A)$, or simply $A$, is a \emph{DG $R$-algebra}. This means that $A = \bigoplus  _{n \geq 0} A _n$ is a non-negatively graded $R$-algebra with a differential $d^A=\{d_n^A\colon A_n\to A_{n-1}\}_{n\geq 0}$ (i.e., each $d_n^A$ is an $R$-linear map with $d_n^Ad_{n+1}^A=0$) such that
for all homogeneous elements $a \in A_{|a|}$ and $b \in A_{|b|}$
\begin{enumerate}[\rm(i)]
\item
$ab = (-1)^{|a| |b|}ba$ and $a^2 =0$  if $|a|$ is odd;
\item
$d^A(ab) = d^A(a) b + (-1)^{|a|}ad^A(b)$, that is, $d^A$ satisfies the \emph{Leibniz rule}.
\end{enumerate}
Here, $|a|$ is called the degree of the homogeneous element $a$.
\end{para}

\begin{para}\label{para20221103a}
DG modules in this paper are \emph{right} DG modules, unless otherwise is specified. By a \emph{DG $A$-module} $(M, \partial^M)$, or simply $M$, we mean a graded right $A$-module $M=\bigoplus_{n\in \mathbb{Z}}M_n$ with a differential $\partial^M=\{\partial^M_n\colon M_n\to M_{n-1}\}_{n\in \mathbb{Z}}$ that satisfies the \emph{Leibniz rule}, i.e., for all homogeneous elements $a\in A_{|a|}$ and $m\in M_{|m|}$ the equality
$\partial^M(ma) = \partial^M(m)\ a + (-1)^{|m|} m\ d^A(a)$ holds. (Left DG $A$-modules are defined similarly.) We often refer to $M=\bigoplus_{i\in \mathbb{Z}}M_i$ as the \emph{underlying graded $A$-module}.


For an integer $i$, the \emph{$i$-th shift} of a DG $A$-module $M$ is denoted by $\shift^i M$. Note that $\left(\shift^i M, \partial^{\shift^i M}\right)$ is a DG $A$-module such that for all integers $j$ we have $\left(\shift^i M\right)_j = M_{j-i}$ and $\partial_j^{\shift^i M}=(-1)^i\partial_{j-i}^M$. We simply write $\shift M$ instead of $\shift^1 M$.
\end{para}

\begin{para}
We denote by $\C(A)$ the category of DG $A$-modules and DG $A$-module homomorphisms. For DG $A$-modules $M,N$, 
the set of DG $A$-module homomorphisms from $M$ to $N$ is denoted by $\Hom_{\C(A)}(M,N)$. Isomorphisms in $\C(A)$ are identified by $\cong$ and quasiisomorphisms in $\C(A)$ (i.e., DG $A$-module homomorphisms that induce bijections on all homology modules) are
identified by $\simeq$.
 
The homotopy category of $A$ is denoted by $\K(A)$, which is a triangulated category. Objects of $\K(A)$ are DG $A$-modules and morphisms are the set of homotopy equivalence classes of DG $A$-module homomorphisms.
In fact, for DG $A$-modules $M,N$ we have $\Hom _{\K(A)} (M, N)= \Hom _{\C(A)} (M, N)/ \sim$. Note that for $f,g\in \Hom _{\C(A)} (M, N)$ we write $f \sim g$ if and only if there exists a graded $A$-module homomorphism $h\colon M \to \shift^{-1}N$ of underlying graded $A$-modules such that the equality $f - g = \partial ^N h + h \partial ^M$ holds. 
\end{para}

\begin{para}
A DG $A$-module $P$ is called \emph{semifree} if it has a \emph{semfree basis}, i.e., if there is a well-ordered subset $F\subseteq P$ which is a basis for the underlying graded $A$-module $P$ such that for every $f\in F$ we have $\partial^P(f)\in \sum_{e<f}eA$. Equivalently, the DG $A$-module $P$ is semifree if there exists an increasing filtration $$0=P_{-1}\subseteq P_0\subseteq P_1\subseteq \cdots\subseteq P$$ of DG $A$-submodules of $P$ such that $P=\bigcup_{i\geq 0}P_i$ and each DG $A$-module $P_i/P_{i-1}$ is a direct sum of copies of $\shift^n A$ with $n\in \mathbb{Z}$; see~\cite{AH},~\cite[A.2]{AINSW}, or~\cite{felix:rht} for details. 

A \emph{semifree resolution} of a DG $A$-module $M$ is a quasiisomorphism $P\xra{\simeq}M$, where $P$ is a semifree DG $A$-module.

For $i\in \mathbb{Z}$ and DG $A$-modules $M,N$, the extension $\Ext^i_A(M,N)$ is defined to be $\HH_{-i}\left(\Hom_A(P,N)\right)$, where $P\xra{\simeq}M$ is a semifree resolution of $M$. Note that $$\Ext^i_A(M,N)=   \Hom _{\K(A)}(P, \shift^iN).$$

\end{para}

\begin{para}
A DG $A$-module $M$ is called {\it perfect} if it is quasiisomorphic to a semifree DG $A$-module that has a finite semifree basis. Equivalently, a DG $A$-module $M$ is perfect if there is a finite sequence of DG $A$-submodules
$$
0=M_0\subseteq M_1\subseteq \cdots \subseteq M_{n-1} \subseteq M_{n} = M
$$
of $M$ such that each $M_k/M_{k-1}$ is a finite direct sum of copies of $\shift^{r_k}A$ with $r_k\geq 0$.
Note that a direct summand of a perfect DG $A$-module is not necessarily perfect.



\end{para}

\begin{para}\label{para20230425r}
Throughout the paper, $(B,d^B)$, or simply $B$, is another DG $R$-algebra that is semifree as a DG $A$-module and $\varphi\colon A\to B$ is a homomorphism of DG $R$-algebras.
Note that $A$ is a DG $R$-subalgebra of $B$, and setting $\overline{B}= B/A$, we see that $\overline{B}$ is a semifree DG $A$-module as well. We further assume that $\overline{B}$ is positively graded, i.e., $B_n=0$ for all $n\leq 0$.

Examples of DG $R$-algebras $A$ and $B$ that satisfy these conditions include the polynomial and free extensions $B=A[X_i\mid 1\leq i\leq n]$ and $B=A\langle X_i\mid 1\leq i\leq n\rangle$ of $A$ with $n\leq \infty$ and $|X_i|>0$ for all $1\leq i\leq n$, where in the latter case $A$ is assumed to be a divided power DG $R$-algebra; see~\cite[Example 2.5]{NOY2}.
\end{para}


\begin{para}
Let $B^e=B \otimes_A B$, or more precisely $(B^e, d^{B^e})$, be the \emph{enveloping DG $R$-algebra of $B$ over $A$}.
The algebra structure on $B^e$ is given by the formula $$(b_1 \otimes_A b_2)( {b'}_1 \otimes_A {b'}_2) =  (-1)^{|{b'}_1| |b_2|} b_1 {b'}_1 \otimes_A b_2{b'}_2$$
for all homogeneous elements $b_1, b_2, b'_1, b'_2\in B$ and the graded structure is given by the equalities $(B^e)_i  = \sum_{j} B_j\otimes_R B_{i-j}$ for all $i\geq 0$ with the relations $b_j\otimes_R a b_{i-j} = b_j a\otimes_R b_{i-j}$ for all $b_j \in B_j$, $b_{i-j}\in B_{i-j}$, and $a\in A$.
Note that $B$ is considered as a DG $R$-subalgebra of $B^e$ via the injective DG $R$-algebra homomorphism $B \to B^e$ defined by $b\mapsto b \otimes_A 1$.

Following~\cite{NOY1}, the kernel of the DG $R$-algebra homomorphism $\pi_B\colon B^e\to B$ defined by $\pi_B(b\otimes_A b')=bb'$ is called the \emph{diagonal ideal} and is denoted by $J$. Hence, we have a short exact sequence
\begin{equation}\label{basic sequence}
0 \to J \xra{\iota} B^e \xra{\pi_B}  B\to 0
\end{equation}
of DG $B^e$-modules in which $\iota$ is the natural injection.
\end{para}

\begin{para}
For a DG $B$-ideal $I$ and an integer $n \geq 0$, let $I^{\otimes _B n}=I \otimes _ B \cdots \otimes _ B I$ be the $n$-fold tensor product of $I$ over $B$ with the convention that $I^{\otimes_B 0}=B$.

Note that $J$ is free as an underlying graded left $B$-module. Hence, applying the functor $- \otimes _ B J^{\otimes _B n}$ to~\eqref{basic sequence}, we obtain the short exact sequence
\begin{equation}\label{eq20230113a}
0 \to J ^{\otimes _B (n+1)}\xra{\iota_n}  B \otimes _A J^{\otimes _B n} \xra{\pi_n} J^{\otimes _B n} \to 0
\end{equation}
of DG $B^e$-modules.
\end{para}

\begin{para}
The $A$-linear map $\delta\colon B\to J$ defined by $\delta(b)=b\otimes_A 1-1\otimes_A b$ for all $b\in B$ is well-defined and is called the \emph{universal derivation}; see~\cite{NOY3}. Under the setting of~\ref{para20230425r} we see that $\inf J:=\inf\{n\mid J_n\neq 0\}>0$; see~\cite[Proposition 2.9]{NOY2}.
\end{para}


The rest of this section is devoted to the tensor algebra and some of its properties.

\begin{para}
Let $T=\bigoplus _{n \geq 0} (\shift J) ^{\otimes_B n}$ be the \emph{tensor algebra of $\shift J$ over $B$} (or the \emph{diagonal tensor algebra}).
The $R$-algebra structure of $T$ comes from taking the tensor product over $B$ as the multiplication. Note that $T$ is also a DG $B^e$-module with the differential $\partial^T$ defined by the Leibniz rule on $(\shift J) ^{\otimes_B n}$ for all $n\geq 1$ as follows:
$$
\partial ^T (c_1 \otimes_B \cdots \otimes_B c_n ) = \sum _{i=1}^n (-1) ^{|c_1| + \cdots + |c_{i-1}|} c_1 \otimes_B \cdots \otimes_B \partial ^{\shift J} (c_i) \otimes_B \cdots \otimes_B c_n
$$
where  $c_i \in \shift J$ for all $1 \leq i \leq n$.

Throughout the paper, we simply write  $T$ instead of $(T, \partial^T)$, which we regard as a DG $R$-algebra (but not necessarily commutative), and at the same time,
as a DG $B^e$-module.
\end{para}


\begin{para}\label{para20230114a}
By definition, a DG $T$-module $M$ is an underlying graded $B$-module $M$ with a DG $B$-module homomorphism $\mu\colon M \otimes_B \shift J \to M$,  which defines the action of $T$ on $M$ as
$m \cdot c := \mu (m \otimes_B c)$, for all $m \in M$ and $c \in \shift J$.
Note that  $$T^+ = T \otimes _B \shift J = \bigoplus _{n\geq 1} (\shift J)^{\otimes _B n}$$  is  a two-sided DG ideal of $T$.
The natural inclusion $T^+\hookrightarrow T$ defines a DG $T$-module structure on $T$.
Note also that $T/T^+ \cong B$ and hence, every DG $B$-module is naturally a DG $T$-module.
\end{para}

\begin{para}\label{para20230114b}
In addition to the usual grading, the DG $R$-algebra $T$ has another grading, which we refer to as the \emph{tensor grading}. More precisely, we say that an element $t\in T$ is of tensor degree $n$ if and only if $t$ belongs to $(\shift J) ^{\otimes _B n}$.
If necessary, we call the ordinary grading that gives the DG $R$-algebra structure on $T$ the \emph{DG grading}.
Hence,  the $n$-th tensor graded part of $T$ is $(\shift J)^{\otimes _B n}$, which we denote by $T^n$, i.e., in the tensor grading we have $T= \bigoplus _{n\geq 0} T^n$.
On the other hand,  the $n$-th DG graded part of $T$ is
$$B_n \oplus (\shift J)_{n} \oplus \left((\shift J)^{\otimes _B 2}\right)_{n} \oplus \cdots=B_n \oplus J_{n-1} \oplus (J^{\otimes _B 2})_{n-2} \oplus \cdots.$$
\end{para}

\section{Tensor graded DG modules}\label{sec20230425d}

The purpose of this section is to introduce a specific map $\omega_N$, where $N$ is a semifree DG $B$-module, which plays a crucial role in this paper; see~\ref{para20230115e} for the definition. We start with generalizing our discussion in~\ref{para20230114b} by defining the notions of tensor graded DG $T$-modules and tensor graded DG $T$-module homomorphisms. 


\begin{para}\label{para20230114c}
A \emph{tensor graded DG $T$-module} $M$ is a direct sum  $M = \bigoplus _{i \in \mathbb{Z}} M^i$ of DG $B$-modules with a collection $\{a_M ^i\colon M^i \otimes _B \shift J \to M^{i+1}\}_{i \in \mathbb{Z}}$ of DG $B$-module homomorphisms.
Notice that this collection $\{ a_M ^i\}_{i\in \mathbb{Z}}$ defines an action of $T$ on $M$ and makes it into a DG $T$-module; see~\ref{para20230114a}.

Let $N = \bigoplus _{i \in \mathbb{Z}} N^i$ be another tensor graded DG $T$-module. A \emph{tensor graded DG $T$-module homomorphism}  $f\colon M \to N$  is a collection $\{ f^i\colon M^i\to N^i\}_{i \in \mathbb{Z}}$ of DG $B$-module homomorphisms which make the following squares commutative in $\C(B)$:
$$
\xymatrix{
M^i \otimes _B \shift J \ar[r]^-{a_M^i}  \ar[d] _-{f^i \otimes_B \id_{\shift J} } & M^{i+1} \ar[d] ^-{f^{i+1}} \\
N^i \otimes _B \shift J \ar[r]^-{a_N^i} &N ^{i+1}.
}
$$
We denote by $\C ^{gr}(T)$  the category of all tensor graded DG $T$-modules and tensor graded DG $T$-module homomorphisms.
\end{para}

\begin{para}
One can check that $\C^{gr}(T)$ is an abelian category. Note that $\C^{gr}(T)$ possesses two shift functors that are of different natures.
For an object $M$ in $\C^{gr}(T)$ and an integer $n$, we have $\shift^n M$, which is the $n$-th shift that is defined in~\ref{para20221103a} using the DG $T$-module structure of $M$.
On the other hand, using the notation from~\ref{para20230114c}, the $n$-th shift of $M$ as a tensor graded DG $T$-module is defined to be $M[n]= \bigoplus_{i\in \mathbb{Z}} M^{i+n}$.
Note that both of the functors $\shift^n$ and $[n]$ are auto-functors on $\C^{gr}(T)$.
\end{para}

\begin{para}
Let  $M$  be in $\C^{gr}(T)$ and use the notation from~\ref{para20230114c}.
For each integer $i$, the $i$-th tensor graded part $M^i$ is a DG $B$-submodule of $M$, and hence, $\partial^M (M^i) \subseteq M^i$.
If $m \in M^i$ and $c \in T^1=\shift J$, then $m\cdot c=a_M^i(m\otimes_Bc) \in M^{i+1}$. It then follows from the Leibniz rule that
$$M^{i+1}\ni\partial^M (m \cdot c) = \partial^M(m) \cdot c \pm m \cdot d^T(c).$$
\end{para}

\begin{ex}
Let $\mathbb{B}= B \otimes _A T$ and consider the DG $B^e$-module $(\mathbb{B}, \mathbb{D})$ constructed in~\cite[\S 3]{NOY2}.
Note that the direct sum description $\mathbb{B} = \bigoplus_{i\in \mathbb{Z}} (B \otimes _A T^i)$ does not give a tensor graded structure on $(\mathbb{B}, \mathbb{D})$  because by the definition of $\mathbb{D}$ we have $\mathbb{D}(B \otimes _A T^i) \subseteq  (B \otimes _A T^i) \cup (B \otimes _A T^{i-1})\not\subseteq  B \otimes _A T^i$. Hence, $(\mathbb{B}, \mathbb{D}) \not\in \C ^{gr}(T)$.
On the other hand, $\partial ^{B \otimes _A T} (B \otimes _A T^i) \subseteq  B \otimes _A T^i$ for all integers $i$ and we see that $(\mathbb{B}, \partial ^{B \otimes _A T})\in \C^{gr}(T)$.
\end{ex}

\begin{ex}
If $M$ is a DG $B$-module, then $M \otimes _B T$ is a tensor graded DG $T$-module by considering $(M\otimes_BT)^i:=M\otimes_BT^i$ with $a^i_{M\otimes_BT}:=\id_{M\otimes_BT^{i+1}}$ for all $i\geq 0$. (Here, we assume that $(M\otimes_BT)^i=0$ for all integers $i<0$.) 
\end{ex}

\begin{para}
The homotopy category $\K^{gr} (T)$ is defined in a natural way. In fact, the objects of $\K^{gr}(T)$ are the same objects as $\C^{gr} (T)$, i.e., all tensor graded DG $T$-modules. Also, assuming $M,N$ are tensor graded DG $T$-modules, we set
$$
\Hom _{\K^{gr} (T)} (M,N)= \Hom _{\C^{gr} (T)} (M,N) /\sim
$$
where  $\sim$ denotes the homotopy, i.e., for $f,g\in \Hom _{\C^{gr} (T)} (M,N)$ we have $f \sim g$ if and only if there exists a tensor graded $T$-linear map $h\colon M \to \shift^{-1} N$ such that $f-g = h \partial^M + \partial ^N h$.
It is straightforward to check that $\K ^{gr}(T)$ is a triangulated category with the shift functor of the DG grading, i.e., with $\shift$.

The derived category $\D^{gr}(T)$ is also defined naturally by taking the localization of  $\K ^{gr}(T)$ by the multiplicative system of all quasiisomorphisms.
\end{para}

\begin{para}
For tensor graded DG $T$-modules $M,N$ let
\begin{gather*}
{}^*\!\Hom _{\K^{gr} (T)}(M,N)= \bigoplus _{n \in \mathbb{Z}} \Hom _{\K^{gr} (T)} (M, N[n])\\
\Gamma _M = {}^*\End _{\K^{gr} (T)}(M)
\end{gather*}
where ${}^*\End _{\K^{gr} (T)}(M)$ denotes ${}^*\!\Hom _{\K^{gr} (T)}(M,M)$. Note that  $\Gamma _M$  is a tensor graded ring in which the multiplication is defined by the composition of morphisms and for an integer $n$, its $n$-th tensor graded part is $\Gamma ^n_M=\Hom _{\K^{gr} (T)}(M,M[n])$.
We call $\Gamma _M = \bigoplus _{n\in \mathbb{Z}} \Gamma _M^n$ the \emph{graded endomorphism ring of $M$ in $\K^{gr} (T)$}.
Note that ${}^*\!\Hom _{\K^{gr} (T)}(M,N)$ is naturally a graded right $\Gamma_M$- and left $\Gamma _N$-module.
\end{para}


The following key lemma will be used frequently in the subsequent sections.

\begin{lem}[Adjointness]\label{lem20230116a}
Assume that $M \in \K(B)$ and $N=\bigoplus_{i\in \mathbb{Z}}N^i \in \K^{gr}(T)$. Then, there exists a natural bijection
$$
\Hom _{\C(B)} (M, N^0) \xra{\zeta} \Hom _{\C^{gr}(T)} (M \otimes _B T, N)
$$
defined by $\zeta(\alpha)(m \otimes_B t)=\alpha (m) t$ for all $\alpha\in \Hom _{\C(B)} (M, N^0)$, $m\in M$, and $t\in T$. Moreover, $\zeta$ induces the bijection
\begin{equation}\label{eq20230604a}
\Hom _{\K(B)} (M, N^0) \xra{\cong} \Hom _{\K^{gr}(T)} (M \otimes _B T, N).
\end{equation}
\end{lem}

\begin{proof}
The inverse map $\zeta^{-1}$ is defined by $\zeta^{-1}(f)=f^0$, where $f^0\colon M \to N^0$ is the restriction of $f\in \Hom _{\C^{gr}(T)} (M \otimes _B T, N)$ to the $0$-th tensor graded part. It is straightforward to see that $\zeta$ and $\zeta^{-1}$ both preserve a homotopy equivalence. Therefore, the isomorphism~\eqref{eq20230604a} follows.
\end{proof}


\begin{para}\label{para20230115d}
Applying the shift functor $\shift^n$ on the short exact sequence~\eqref{eq20230113a} and taking the direct sum $\bigoplus_{n\geq 0}$, since $(T^+)^n=(\shift J)^{\otimes_Bn}$, we obtain a short exact sequence
\begin{equation}\label{eq20230115a}
0 \to \shift^{-1}(T^+ [1])  \to B \otimes _A T \to T \to 0.
\end{equation}
of DG $T$-modules in $\C^{gr}(T)$. 
\end{para}

The following proposition is an immediate consequence of our discussion in~\ref{para20230115d}.

\begin{prop}\label{triangles}
Let $N$  be a semifree DG  $B$-module. Then the following hold.
\begin{enumerate}[\rm(a)]
\item\label{triangles1}
There is a short exact sequence
\begin{equation}\label{eq20130118a}
0 \to N \otimes _B \shift^{-1}(T^+ [1]) \to N \otimes _A T \xra{\theta_N} N \otimes _B T \to 0
\end{equation}
in $\C^{gr}(T)$, where  $\theta_N$  is the induced natural homomorphism.
\item\label{triangles2}
The following triangle exists in $\K^{gr}(T)$:
$$
N \otimes _B \shift^{-1}(T^+ [1])  \to N \otimes _A T \xra{\theta_N} N \otimes _B T \xra{\omega^+_N} N \otimes _B T^+[1].
$$
\item\label{triangles3}
The following triangle exists in $\D^{gr}(T)$:
$$
N \lotimes _B \shift^{-1}(T^+ [1]) \to N \lotimes _A T \xra{\theta_N}  N \lotimes _B T \xra{\omega^+_N} N \lotimes _B T^+[1].
$$
\end{enumerate}
\end{prop}


In the next result, we explicitly describe the connecting morphism $\omega^+_N$.

\begin{lem}\label{lem20231015a}
Let $N$ be a semifree DG $B$-module with a semifree basis $\{ e_{\lambda}\}_{\lambda \in \Lambda}$, and let $\partial ^N (e_{\lambda}) = \sum _{\mu < \lambda} e_{\mu} b_{\mu \lambda}$,
where $b _{\mu \lambda} \in B$.
Then, the tensor graded DG $T$-module homomorphism $\w_N^+\colon N\otimes_BT\to N\otimes_BT^+[1]$ in $\C^{gr}(T)$ defined by
$$
\w ^+_N (e_{\lambda} \otimes _B t  ) = \sum _{\mu < \lambda} e_{\mu} \otimes _B \delta (b_{\mu \lambda} ) \otimes _B t
$$
for all $t\in T$ satisfies the equality $\omega_N^+=[\w_N^+]$.
\end{lem}

\begin{proof}
Consider the right $B$-linear homomorphisms $\sigma_N\colon N\otimes_BB^e\to N\otimes_BJ$ and $\rho_N\colon N \to N\otimes_BB^e$ which are defined by the equalities
\begin{gather*}
\sigma_N(e_{\lambda}\otimes_B (b_1\otimes_A b_2))=e_{\lambda}\otimes_B(b_1\otimes_A b_2-1\otimes_A b_1b_2)\\
\rho_N(e_{\lambda} b)=e_{\lambda}\otimes_B (1\otimes_A b)
\end{gather*}
for a basis element $e_{\lambda}$ and homogeneous elements $b,b_1,b_2\in B$.
We define
$$\w^+_N=(\sigma_N\partial^{N\otimes_BB^{e}}\rho_N)\otimes_B\id_T.$$
The assertion now follows from our computations in~\cite[Theorem 3.9]{NOY3}.
\end{proof}

Next, for a semifree DG $B$-module $N$, we define an element $\omega_N$ in the tensor graded endomorphism ring $\Gamma_{N\otimes _BT}$ that plays a crucial role in this paper.

\begin{para}\label{para20230115e}
Let $N$ be a semifree DG $B$-module with a semifree basis $\mathcal{B}=\{ e_{\lambda}\}_{\lambda \in \Lambda}$. For $e_{\lambda}\in \mathcal{B}$ assume that $\partial ^N (e_{\lambda}) = \sum _{\mu < \lambda} e_{\mu} b_{\mu \lambda}$,
where $b _{\mu \lambda} \in B$. Let $$\w_N\colon N \otimes _B T \to N \otimes _B T [1]$$ be the composition $\varpi_N\w^+_N$ in $\C^{gr}(T)$, where  $\varpi_N\colon N \otimes _B T^+ [1]\to N \otimes _B T[1]$ is the natural inclusion map. Note that for all $e_{\lambda}\in \mathcal{B}$ and $t\in T$ we have
\begin{equation}\label{eq20230423a}
\w_N (e_{\lambda} \otimes_B t) = \sum _{\mu < \lambda} e_{\mu} \otimes _B \delta (b_{\mu \lambda} ) \otimes _B t.
\end{equation}
Therefore, $\omega_N:=[\w_N] \in \Gamma_{N\otimes_BT} ^1= \Hom _{\K^{gr} (T)} (N \otimes _B T, N \otimes _B T[1]) \subseteq \Gamma_{N\otimes _BT}$.

For all integers $\ell\geq 1$, set $\w_N^{\ell}:=\w_N[\ell-1]\circ \cdots \circ \w_N[1]\circ \w_N$. Then, for all basis elements $e_{\lambda}$ and $t\in T$, one can compute that
\begin{equation}\label{eq20230822a}
\w_N^{\ell}(e_{\lambda}\otimes_Bt)=\!\!\!\!\!\!\sum_{\mu_{\ell}<\cdots<\mu_{1}<\lambda}\!\!\!\!\!\!e_{\mu_{\ell}}\otimes_B\delta(b_{\mu_{\ell}\mu_{\ell-1}})\otimes_B\cdots\otimes_B\delta(b_{\mu_{2}\mu_{1}})\otimes_B\delta(b_{\mu_{1}\lambda})\otimes_Bt.
\end{equation}
\end{para}

\begin{para}
The definition of $\w_N$ from~\ref{para20230115e} is independent of the choice of the semifree basis for the semifree DG $B$-module $N$ in the following sense:
Let $\mathcal{B}'=\{e'_\lambda\}_{\lambda\in\Lambda}$ be another semifree basis of $N$ and $u\colon N\to N$ be the DG $B$-module automorphism such that $e'_\lambda=u(e_\lambda)$. Note that $\partial^N(e'_\lambda)=\sum_{\mu<\lambda}e'_{\mu}b_{\mu\lambda}$. Let $\w_N=\w_N^\mathcal{B}$ and consider $\w_N^\mathcal{B'}$ which is defined by a similar formula as~\eqref{eq20230423a}, i.e., $\w_N^{\mathcal{B'}}(e'_\lambda \otimes_B t)= \sum_{\mu<\lambda} e'_{\mu}\otimes_B \delta(b_{\mu\lambda})\otimes_B t$.
Then, we have $\w_N^{\mathcal{B'}}=(u\otimes_B \id_T)\w_N^{\mathcal{B}}u^{-1}$.

Also, it is important to note that $\w_N$ preserves the DG degree and is not a map of DG degree $-1$.
\end{para}

In the following result, for a positive integer $n$, we denote by $B^{\oplus n}$ the direct sum of $n$ copies of the DG $R$-algebra $B$.

\begin{prop}\label{cor20230124a}
Let $n$ be a positive integer. Then $\omega_{B^{\oplus n}}=0$ in $\K^{gr}(T)$.
\end{prop}

\begin{proof}
Note that the short exact sequence~\eqref{eq20130118a} is split for $N=B^{\oplus n}$ and hence, $\omega_{B^{\oplus n}}^+=0$ in $\K^{gr}(T)$. This implies that $\omega_{B^{\oplus n}}=0$ in $\K^{gr}(T)$.
\end{proof}

\begin{prop}\label{lem20230115b}
Assume that $N$ is a semifree DG $B$-module that is bounded below (i.e., $N_i=0$ for all $i\ll 0$).
Then, $\w_N$ is locally nilpotent, that is, for each element $x  \in N \otimes _B T$, there exists a positive integer $n_x$ such that $\w_N^{n_x}(x)=0$, where $\w_N^{n_x}$ is the composition morphism $\w_N[n_x-1]\circ \cdots\circ \w_N[1]\circ \w_N$.
\end{prop}

Note that the locally nilpotent property does not mean that $\w_N$ itself is nilpotent; compare this proposition with Theorem~\ref{equivalent conditions}.

\begin{proof}
Assume that $x\in N\otimes_B T$ is an element of the DG degree $m$ and of the tensor degree $i$, that is,  $x \in (N \otimes _B T^i)_m$.
Then, for each integer $n$ we have $$\w_N ^n(x) \in  (N \otimes _B T^{i+n})_m = (N \otimes _B J^{\otimes_B (i+n)})_{-i-n+m}$$
which vanishes if  $n \gg  m-i$ since the lower bound of $N \otimes _B J^{\otimes_B (i+n)}$ is not less than that of $N$. Hence, such a positive integer $n_x$ with $\w_N ^{n_x}(x)=0$ exists.
\end{proof}

The map $\omega_N$ can also be described functorially as follows.

\begin{prop}\label{lem20230117a}
Let $N,N'$ be semifree DG $B$-modules and $f\colon N \to N'$  be a DG $B$-module homomorphism.
Then, there exists a commutative diagram
\begin{equation}\label{eq20230117a}
\xymatrix{
N \otimes _B T  \ar[r]^-{\omega _N} \ar[d]_{[f\otimes_B \id_T]} & N \otimes _B T[1] \ar[d]^{[f\otimes_B \id_{T[1]}]} \\
N' \otimes _B T  \ar[r]^-{\omega _{N'}}  & N' \otimes _B T[1] \\
}
\end{equation}
in $\K^{gr} (T)$.
\end{prop}

\begin{proof}
There is a commutative diagram
$$
\xymatrix{
0 \ar[r] & N \otimes _B (\shift^{-1}T^+)[1] \ar[rr] \ar[d]^{f \otimes_B \id_{(\shift^{-1}T^+)[1]}} && N \otimes _AT \ar[rr] \ar[d]^{f \otimes_A \id_T} && N \otimes _BT  \ar[r] \ar[d]^{f \otimes _B \id_T}& 0 \\
0 \ar[r] & N' \otimes _B (\shift^{-1}T^+)[1] \ar[rr] && N' \otimes _AT \ar[rr] && N' \otimes _BT  \ar[r] & 0 \\
}
$$
with exact rows in $\C^{gr}(T)$. This induces the commutative diagram
$$
\xymatrix{
N \otimes _B (\shift^{-1}T^+)[1] \ar[r] \ar[d]^{[f \otimes _B \id_{(\shift^{-1}T^+)[1]}]} & N \otimes _AT \ar[r] \ar[d]^{[f \otimes _A \id_T]} & N \otimes _BT  \ar[rr]^-{\omega ^+ _N}  \ar[d]^{[f \otimes _B \id_T]}&& N \otimes _B T^+[1]  \ar[d]_{[f \otimes _B \id_{T^+[1]}]} \\
N' \otimes _B (\shift^{-1}T^+)[1] \ar[r]& N' \otimes _AT \ar[r] & N' \otimes _BT  \ar[rr]^-{\omega ^+_{N'}} && N' \otimes _B T^+[1] \\
}
$$
with triangle rows in $\K ^{gr}(T)$. On the other hand, the diagram
$$
\xymatrix{
N \otimes _B T^+[1] \ar[r] ^{[\varpi_N]} \ar[d] _{[f \otimes _B \id_{T^+[1]}]} & N \otimes _B T[1] \ar[d] ^{[f \otimes _B \id_{T[1]}]} \\
N' \otimes _B T^+[1] \ar[r] ^{[\varpi_{N'}]} & N \otimes _B T[1]
}
$$
is also commutative.
The existence of the commutative diagram~\eqref{eq20230117a} now follows from the equalities $\omega _N=[\varpi_N]\circ \omega ^+_{N}$ and $\omega _{N'}=[\varpi_{N'}]\circ \omega ^+_{N'}$.
\end{proof}

The next result indicates that $\omega_N$ commutates with the elements of $\Gamma _{N \otimes _B T}$.

\begin{cor}\label{cor20230117a}
Let $N$ be a semifree DG $B$-module. If  $f \in \End _{\K(B)} (N)$, then we have the equality
$(f \otimes _B \id_{T[1]}) \omega_N = \omega_N (f \otimes _B\id_T)$ in $\K^{gr}(T)$.
\end{cor}

\section{Hom sets in the homotopy categories}\label{sec20230425s}

In this section, for a semifree DG $B$-module $N$ and integers $m$, we discuss the action of $\shift^m\omega_N$ as an element in the ring $\Gamma_{N \otimes _B T}$ on the graded right $\Gamma_{N \otimes _B T}$-modules ${}^*\!\Hom _{\K^{gr}(T)}(N \otimes _B T, N \otimes _B (\shift^m T))$, assuming that condition $\AR$, defined in~\ref{para20230424v} below, holds for $N$. Our main result in this section is Theorem~\ref{main} whose proof takes up the balance of this section and involves several steps, including an analysis of the mapping cone of $\omega_N$.

\begin{para}\label{para20230424v}
We say that $\AR$ holds for a DG $B$-module $N$ if:
\begin{enumerate}[\rm(i)]
\item\label{AR-1-1}
$N$  is  a semifree DG $B$-module that is non-negatively graded;
\item\label{AR-1-2}
$B$ and $N$ are perfect considered as DG $A$-modules via $\varphi$; and
\item\label{AR-1-3}
$\Hom_{\K(B)}(N, \shift^nB) = 0$ for all integers $n\geq 1$. 
\end{enumerate}
\end{para}

\begin{para}
Note that $\AR$ is not necessarily preserved under degree shift. More precisely, if $\AR$ holds for a DG $B$-module $N$ and $i$ is a non-zero integer, then $\shift^iN$ may not satisfy conditions \eqref{AR-1-1} and \eqref{AR-1-3}. 
\end{para}

\begin{thm}\label{main}
If $\AR$ holds for a DG $B$-module $N$, then the map
\begin{align*}
\Hom_{\K^{gr}(T)}(N\otimes_B T, \shift^m(N\otimes _BT[n])) &\xra{\shift^m(\omega_N[n])\circ-}&\\
&\Hom_{\K^{gr}(T)}(N\otimes_B T, \shift^m(N\otimes _BT [n+1]))&
\end{align*}
which is defined by the left composition by $\shift^m(\omega_N[n])$ is always surjective for all $n, m \geq 0$ and is bijective if either
$m \geq 1$ and $n \geq 0$, or $m =0$ and  $n\geq 1$.
In particular,  the left action of  $\shift^m\omega_N$  on
${}^*\!\Hom  _{\K^{gr}(T)}(N \otimes _B T, N \otimes _B (\shift^m T))$ is surjective for all $m \geq 0$.
\end{thm}

The proof of Theorem~\ref{main} will be given near the end of this section.

\begin{para}\label{para20230116a}
For each integer $n\geq 0$, let $\overline{B} ^{\otimes_A n}= \overline{B} \otimes_A  \cdots \otimes_A  \overline{B}$ be the $n$-fold tensor product of $\overline{B}$ over $A$ with the convention that $\overline{B} ^{\otimes_A 0}=A$.
Note that if $B$ is a perfect DG $A$-module via $\varphi$, then $\overline{B}$ is quasiisomorphic to a semifree DG $A$-module with a finite semifree basis that consists of basis elements with positive DG degrees. Hence, for each integer $n\geq 0$, the DG $A$-module $\overline{B}^{\otimes _A n}$ is quasiisomorphic to a semifree DG $A$-module with a finite semifree basis that consists of basis elements of DG degrees $\geq n$.
It follows from \cite[Proposition 2.9]{NOY2} that $J ^{\otimes _B n}$  is also quasiisomorphic to a semifree DG $B$-module with a finite semifree basis consisting of basis elements of degrees $\geq n$.
\end{para}

\begin{lem}\label{first vanishing}
If $\AR$ holds for a DG $B$-module $N$, then
\begin{equation}\label{eq20230124a}
\Hom _{\K(B)}(N, \shift^m(J^{\otimes _B n})) = 0
\end{equation}
for all integers $n \geq 0$ and $m >  -n$.
\end{lem}

\begin{proof}
By our discussion in~\ref{para20230116a} there is a finite filtration
$$
0=F_0 \subseteq F_1 \subseteq F_2 \subseteq \cdots \subseteq F_i \simeq J^{\otimes _B n}
$$
of DG $B$-submodules such that each $F_k/F_{k-1}$ is a finite direct sum of copies of $\shift^{r_k}B$ with $r_k \geq n$.
Condition $\AR$\eqref{AR-1-3} implies that $\Hom _{\K(B)} (N, \shift^m(\shift^{r_k}B))=0$ for all $m > -n \geq -r_k$.
Therefore, by induction on $k$ we have $\Hom _{\K(B)} (N, \shift^mF_k)=0$ for all $1\leq k\leq i$ and $m > -n$. In particular, the equality~\eqref{eq20230124a} holds.
\end{proof}

\begin{lem}\label{vanishing}
If $\AR$ holds for a DG $B$-module $N$, then
\begin{enumerate}[\rm(a)]
\item
$\Hom _{\K(B)}(N, N \otimes _A \shift^m(J^{\otimes _B n})) = 0$
for all $n \geq 0$ and $m > -n$.

\item
$\Hom _{\K^{gr}(T)} (N \otimes _B T,  N \otimes _A \shift^m(T[n])) =0$ for all $n \geq 0$ and $m > -2n$.
\end{enumerate}
\end{lem}

\begin{proof}
(a) By $\AR$\eqref{AR-1-2}, there is a finite filtration
$$
0=F_0 \subseteq F_1 \subseteq F_2 \subseteq \cdots \subseteq F_i \simeq N
$$
of DG $A$-submodules, where each $F_k/F_{k-1}$ is a finite direct sum of copies of $\shift^{r_k}A$ with $r_k \geq 0$.
Thus, fixing an integer $n \geq 0$,  the DG $B$-module $N \otimes _A J^{\otimes _B n}$ has a finite filtration
$$
0=F'_0 \subseteq F'_1 \subseteq F'_2 \subseteq \cdots \subseteq F'_i \simeq N\otimes _A J^{\otimes _B n}
$$
of DG $B$-submodules, where
$F'_k = F_k \otimes _A  J^{\otimes _B n}$ with each $F'_k/F'_{k-1}$ being a finite direct sum of copies of $\shift^{r_k}(J^{\otimes _B n})$ with $r_k \geq 0$.
Lemma~\ref{first vanishing} implies that
$\Hom _{\K(B)} (N, \shift^m(F'_k/F'_{k-1}))=0$ for all $1\leq k\leq i$ and $m >-n$.
Now, by induction on $1 \leq k \leq i$, we see that  $\Hom _{\K(B)} (N, \shift^m F'_k)=0$  for  all $m >-n$, as well.

(b) By Lemma~\ref{lem20230116a}, we only need to show that
$\Hom _{\K(B)} (N,  N \otimes _A \shift^mT^n) =0$ for all $n \geq 0$ and $m > -2n$.
Equivalently, we need to show that the equality
$\Hom _{\K(B)} (N, N \otimes _A \shift^{n+m}(J^{\otimes _B n})) =0$ holds for all $n \geq 0$ and $n+m > -n$. However, this follows from part (a) and the proof is complete.
\end{proof}



\begin{para}
For a semifree DG $B$-module $N$, the mapping cone $C(\omega_N)$ of $\omega_N$ is an object in $\K^{gr}(T)$.
Moreover, the following triangle exists in $\K^{gr}(T)$:
\begin{equation}\label{cone triangle}
\shift^{-1}C(\omega_N) \to N \otimes _B T  \xra{\omega_N} N \otimes _B T[1]  \to C(\omega_N).
\end{equation}
\end{para}

\begin{lem}\label{cone}
Let $N$ be a semifree DG $B$-module. Then, there is a triangle
\begin{equation}\label{cone2}
\shift^{-1}(N[1]) \to N \otimes _A\shift T \to C(\omega_N) \to N[1]
\end{equation}
in $\K^{gr}(T)$, where the DG $B$-module $N$ is regarded as a tensor graded DG $T$-module that is concentrated in tensor degree zero.
In particular, for each integer $n$, taking the $n$-th tensor graded part, we have the following isomorphisms in $\K (B)$:
$$
 C(\omega_N) ^n \cong
 \begin{cases}
  N \otimes _A\shift T^n = N \otimes _A\shift^{n+1}(J^{\otimes _Bn})  &\ \text{if}\ n \geq 0\\
 N &\ \text{if}\ n= -1 \\
 0 &\ \text{if}\ n\leq -2.
 \end{cases}
 $$
\end{lem}

\begin{proof}
By the octahedron axiom, there exists a commutative diagram
$$
\xymatrix{ & N \otimes _B \shift T \ar@{=}[r] & N \otimes _B \shift T& \\
\shift^{-1}(N[1]) \ar[r] & N \otimes _A\shift T \ar[u] \ar[r] & C(\omega_N)  \ar[r] \ar[u] & N[1] \\
\shift^{-1}(N[1]) \ar@{=}[u]  \ar[r] & N \otimes _B T^+ [1] \ar[r]^{[\varpi_N]} \ar[u] & N \otimes _B T [1] \ar[u] \ar[r] &N[1] \ar@{=}[u] \\
&N \otimes _B T \ar[u]^-{\omega_N^+}   \ar@{=}[r] & N \otimes _B T \ar[u]^-{\omega_N} & }
$$
in which all rows and columns are triangles in $\K^{gr}(T)$ and the second row is~\eqref{cone2}.
The last assertion is obtained by taking the $n$-th tensor graded part in~\eqref{cone2} and noting that
$N^n= 0$ for all $n \not= 0$ and $N \otimes _A \shift T^n =0$ for all $n < 0$.
\end{proof}

\begin{prop}\label{cone vanishing}
If $\AR$ holds for a DG $B$-module $N$, then
$$
\Hom _{\K^{gr}(T)} (N \otimes _BT,  \shift^m (C(\omega_N)[n]))\!\cong\!
\begin{cases}
0&\!\!\!\!\!\!\!\!\!\!\!\!\!\!\!\!\!\!\!\!\!\text{if}\ n\geq 0, m \geq -2n, \text{or}\ n\leq -2\\
\Hom _{\K(B)}(N, \shift^m N)&\! \text{if}\ n= -1.
\end{cases}
$$
\end{prop}

\begin{proof}
By Lemmas~\ref{lem20230116a} and~\ref{cone} we have
\begin{eqnarray*}
\Hom _{\K^{gr}(T)} (N \otimes _BT, \shift^m(C(\omega_N)[n]))\!\!\!\!
&\cong&\!\!\!\! \Hom _{\K(B)} (N, \shift^m(C(\omega_N)^n))  \\
&\cong&\!\!\!\!
\begin{cases}
\Hom _{\K(B)} (N, \shift^{m+1}(N \otimes _A T^n)) & \text{if}\ n \geq 0 \\
\Hom _{\K(B)} (N, \shift^m N) & \text{if}\ n=-1 \\
0 & \text{if}\ n \leq -2.
\end{cases}
\end{eqnarray*}
Now, the assertion follows from Lemma~\ref{vanishing}.
\end{proof}


\noindent \emph{Proof of Theorem~\ref{main}.}
Applying the functor $\Hom_{\K^{gr}(T)} (N \otimes _BT, \shift^m ((-)[n]))$ to the triangle~\eqref{cone triangle} we obtain an exact sequence
\begin{align*}
\Hom_{\K^{gr}(T)}(N\otimes_B T, \shift^{m-1}(C(\omega_N)[n])) & \to \Hom_{\K^{gr}(T)}(N\otimes_B T, \shift^m(N\otimes _BT[n]))\\ &\!\!\!\!\!\!\!\!\!\!\!\!\!\!\!\!\!\!\!\!\!\!\!\!\!\!\!\!\!\! \xra{\shift^m(\omega_N[n])\circ-}
\Hom_{\K^{gr}(T)}(N\otimes_B T, \shift^m(N\otimes _BT [n+1]))\\
&\to \Hom_{\K^{gr}(T)}(N\otimes_B T, \shift^m(C(\omega_N)[n]))
\end{align*}
of $R$-modules. By Proposition~\ref{cone vanishing}, if $n \geq 0$, then $\ker(\shift^m(\omega_N[n])\circ-)=0$ for all $m \geq 1$ and $\coker(\shift^m(\omega_N[n])\circ-)=0$ for all $m \geq 0$. Moreover, if $n\geq 1$, then $\ker(\shift^m(\omega_N[n])\circ-)=0$ and $\coker(\shift^m(\omega_N[n])\circ-)=0$ for all $m \geq 0$.
\qed
\vspace{6pt}

We conclude this section with Theorem~\ref{thm20230425a} in which in addition to $\AR$ we are considering another Ext vanishing condition on DG modules, defined next.

\begin{para}\label{para20230425a}
We say that $\Ar$ holds for a DG $B$-module $N$ if $\Hom _{\K(B)}(N, \shift^n N) = 0$ for all integers $n \geq 1$.
\end{para}

\begin{thm}\label{thm20230425a}
If $\AR$ and $\Ar$ hold for a DG $B$-module $N$, then
$$
{}^*\!\Hom _{\K^{gr}(T) } (N \otimes _B T, N \otimes _B \shift^m T)  =0
$$
for all integers $m\geq 1$.
\end{thm}

\begin{proof}
Recall that ${}^*\!\Hom _{\K^{gr}(T) } (N \otimes _B T, N \otimes _B \shift^m T)$ is a non-negatively graded module over the graded ring $\Gamma_{N\otimes _B T}$ with
$$
{}^*\!\Hom _{\K^{gr}(T) } (N \otimes _B T, N \otimes _B \shift^m T)^0\cong \Hom _{\K(B)}(N, \shift^m N)
$$
by Lemma~\ref{lem20230116a}.
Assuming $m\geq 1$, for an integer  $n\geq 1$, by Theorem~\ref{main} we have
$$
{}^*\!\Hom _{\K^{gr}(T) } (N \otimes _B T, N \otimes _B \shift^m T)^n=\omega_N ^n \cdot \Hom _{\K(B)}(N, \shift^m N).
$$
Now the assertion follows from the assumption that $\Ar$ holds for $N$.
\end{proof}

\section{The endomorphism rings}\label{sec20230425r}

In this section, for a semifree DG $B$-module $N$ that satisfies $\AR$, we study the relationship between the endomorphism rings $\End_{\K(B)}(N)$ and $\Gamma_{N\otimes_BT}$. Our main results in the section are Theorems~\ref{omega generates end} and~\ref{kernel of omega} which play an important role in Section~\ref{sec20230422a}.



Our first main result of this section, stated next, shows that if $\AR$ holds for a DG $B$-module $N$, then $\Gamma _{N \otimes _B T}$ is an $\End _{\K(B)}(N)$-algebra generated by $\omega_N$, and if $\omega_N\neq 0$, then it is a non-zero divisor on $\Gamma _{N \otimes _B T}^+=\bigoplus _{n\geq 1} \Gamma _{N \otimes _B T}^n$.

\begin{thm}\label{omega generates end}
If $\AR$ holds for a DG $B$-module $N$, then we have the equality
$$\Gamma _{N \otimes _B T}  = \End _{\K(B)}(N) [\omega_N]$$
of rings, that is, every element of $\Gamma _{N \otimes _B T}$ can be written as a polynomial of the form
$\alpha_0 + \omega_N\alpha_1+ \omega_N ^2 \alpha_2+ \cdots +  \omega_N ^n\alpha_n$, where $\alpha _i \in \End_{\K^{gr(T)}}(N\otimes_BT)\cong\End _{\K(B)}(N)$ for all $0\leq i\leq n$.
Moreover, the left multiplication by $\omega_N$ induces the bijections
$\Gamma _{N \otimes _B T} ^n \xra{\omega_N \cdot-}\Gamma _{N \otimes _B T}^{n+1}$ for all $n\geq 1$.
\end{thm}

The proof of Theorem~\ref{omega generates end} will be given after the following result, which is an immediate consequence of Lemma~\ref{lem20230116a}.

\begin{prop}\label{cor to adjoint}
Let $N \in \K (B)$.
Then, there is an isomorphism
$$
\Gamma _{N\otimes_BT} ^0 \cong \End _{\K(B)} (N)$$
and
$\Gamma ^n _{N \otimes _BT} =0$  for all integers  $n<0$.
In particular,  $\Gamma _{N \otimes _BT}$  is a non-negatively graded ring and we have
$$
\Gamma _{N \otimes _BT} \cong
\bigoplus _{n\geq 0} \Hom_{\K(B)}(N, N \otimes _B (\shift J)^{\otimes_B n})
= \bigoplus _{n\geq 0} \Ext _B^n (N, N \otimes _B J^{\otimes_B n}).
$$
\end{prop}

\noindent \emph{Proof of Theorem~\ref{omega generates end}.}
By Proposition~\ref{cor to adjoint} we know that $\Gamma _{N \otimes _B T}$ is a non-negatively graded ring whose $0$-th graded part is $\End _{\K(B)}(N)$. Also, recall from~\ref{para20230115e} that $\omega_N\in \Gamma^1_{N \otimes _B T}$.
By Theorem~\ref{main}, for all $n\geq 0$, the $n$-th graded part of $\Gamma _{N \otimes _B T}$ is of the form $\omega_N ^n \cdot \End _{\K(B)}(N)$.
Hence, we have the equality
\begin{equation}\label{direct decomp}
\Gamma _{N \otimes _B T} = \bigoplus _{n \geq 0} \omega_N ^n \cdot \End _{\K(B)}(N).
\end{equation}
The last statement of this result follows from Theorem~\ref{main} as well.
\qed


\begin{cor}\label{thm20230117a}
If $\AR$ holds for a DG $B$-module $N$, then $\omega_N$ lies in the center of the ring $\Gamma _{N \otimes _B T}$.
\end{cor}

\begin{proof}
By Theorem~\ref{omega generates end}, we have $\Gamma _{N \otimes _B T}  = \End _{\K(B)}(N) [\omega_N]$. Now, the assertion follows from Corollary~\ref{cor20230117a}.
\end{proof}

The rest of this section is devoted to the second main result of this section, i.e., Theorem~\ref{kernel of omega}. In this theorem, for a DG $B$-module $N$ that satisfies $\AR$, we determine the kernel and cokernel of the map $\Gamma _{N \otimes _B T} \xra{\omega_N \cdot-} \Gamma _{N \otimes _B T}[1]$ induced by the left multiplication by $\omega_N$. For this purpose, we start with the following lemma.

\begin{lem}\label{lem20230118a}
Let $N$ be a DG $B$-module and $\p$ be the ideal of the ring $\End _{\K(B)}(N)$ consisting of all morphisms that factor through a finite direct sum of copies of $B$.
Considering $\p$ as a subset of the ring $\Gamma _{N\otimes _B T}$ via the natural inclusion map $\End _{\K(B)}(N) \hookrightarrow \Gamma _{N\otimes _B T}$ we have the equality $\omega_N \cdot \p =0$.
\end{lem}

\begin{proof}
Any morphism $f \in \p$ is, by definition,  a composition  $N \xra{h} B ^{\oplus n}\xra{g} N$ for some integer $n \geq 1$.
Thus, we obtain a commutative diagram
$$
\xymatrix{
N \otimes _BT \ar[rrrr]^-{f\otimes_B \id_T} \ar[rrd]_{h\otimes_B \id_T}  \ar[ddd]_{\omega_N} &&&& N \otimes _BT  \ar[ddd]^-{\omega_N} \\
&& B^{\oplus n} \otimes _B T\ar[rru]_{g \otimes_B \id_T}  \ar[ddd]^-(0.3){\omega_{B^{\oplus n}} }&& \\ \\
N \otimes _BT[1] \ar[rrrr]^(0.3){f\otimes_B \id_{T[1]}} \ar[rrd]_{h\otimes_B \id_{T[1]}}  &&&& N \otimes _BT[1] \\
&& B^{\oplus n} \otimes _B T[1]\ar[rru]_{g \otimes_B \id_{T[1]}}  & \\
}
$$
in which, by Proposition~\ref{cor20230124a}, we have $\omega _{B ^{\oplus n}} =0$.
Therefore, $\omega_N (f \otimes _B \id_T) =0$, which means that  $\omega_N \cdot \p = 0 $  in $\Gamma _{N\otimes _B T}$.
\end{proof}

\begin{cor}\label{cor20230118a}
If $\AR$ holds for a DG $B$-module $N$, then the ideal $\p$ introduced in Lemma~\ref{lem20230118a} is a graded (two-sided) ideal of the ring $\Gamma _{N\otimes _B T}$ that is concentrated in degree $0$.
\end{cor}

\begin{proof}
By Theorem~\ref{omega generates end}, the ring $\Gamma _{N\otimes _B T}$ is generated by $\omega_N$ as an $\End _{\K(B)}(N)$-algebra. Hence, the ideal $\p$ of $\End _{\K(B)}(N)$ is an ideal of $\Gamma _{N\otimes _B T}$ as well.
\end{proof}

\begin{para}\label{par20230117a}
Assume that $\AR$ holds for a DG $B$-module $N$. By Theorem~\ref{omega generates end}, each morphism $\Gamma _{N \otimes _B T} ^n \xra{\omega_N \cdot-}\Gamma _{N \otimes _B T}^{n+1}$ with $n \geq 1$ is an isomorphism. On the other hand, by Theorem~\ref{main} the map $\Gamma _{N \otimes _B T}^{0} \xra{\omega_N \cdot-} \Gamma _{N \otimes _B T}^{1}$ is surjective, but it is not injective in general. Therefore, it has a non-trivial kernel.
By Lemma~\ref{lem20230116a}, note that the map $\Gamma _{N \otimes _B T}^{0} \xra{\omega_N \cdot-} \Gamma _{N \otimes _B T}^{1}$ is in fact the map
$$
\End _{\K(B)} (N) \to \Hom _{\K(B)}(N, N\otimes_B\shift J)
$$
which is induced by applying the functor $\Hom _{\K(B)} (N , -)$ to $\omega_N\colon N \to N \otimes _B \shift J$. To be precise, the morphism $\omega_N\colon N \to N \otimes _B \shift J$, which we are considering here, is in fact the restricted morphism $\omega_N|_N$, for which we use the same notation $\omega_N$.
Since there is a triangle  $N \otimes _B J \to N \otimes _A B \to N \xra{\omega_N} N \otimes _B \shift J$ in $\K(B)$ which is obtained from the short exact sequence~\eqref{basic sequence}, we have $$\ker(\Gamma _{N \otimes _B T}^{0} \xra{\omega_N \cdot-} \Gamma _{N \otimes _B T}^{1})=\im\left(\Hom _{\K(B)}(N, N \otimes _A B) \to \End _{\K(B)}(N)\right).$$
\end{para}

\begin{lem}\label{lem20230118b}
If $\AR$ holds for a DG $B$-module $N$, then for the ideal $\p$ introduced in Lemma~\ref{lem20230118a}
we have a short exact sequence
$$
0 \to \p \to \End _{\K(B)}(N) \xra{\Hom_{\K(B)}(N, \omega_N)} \Hom _{\K(B)}(N, N\otimes_B \shift J) \to 0.
$$
\end{lem}

\begin{proof}
The map $\Hom_{\K(B)}(N, \omega_N)$ is the multiplication (i.e., composition) by $\omega_N$ from the left. Hence, by Lemma~\ref{lem20230118a} we have $\p \subseteq \ker(\Hom_{\K(B)}(N, \omega_N))$.

For the reverse containment, let $f \in \ker(\Hom_{\K(B)}(N, \omega_N))$ be an arbitrary element.
By~\ref{par20230117a}, we have $f\in \im\left(\Hom _{\K(B)}(N, N \otimes _A B) \to \End _{\K(B)}(N)\right)$.
Note that $N\otimes _AB$ is a perfect DG $B$-module, and hence, it has a finite filtration
$$
0=F_{0} \subseteq F_1 \subseteq \cdots \subseteq F_n \simeq N\otimes _A B
$$
such that each $F_k /F_{k-1}$ is a direct sum of copies of $\shift^{r_k}B$ with $r_k\geq k-1$.
Since  $\Hom _{\K(B)} (N, \shift^i B)=0$ for all $i\geq 1$, we see that the natural inclusion $F_1 \hookrightarrow N \otimes _A B$ induces a surjective map $\Hom _{\K(B)} (N, F_1) \to \Hom _{\K(B)} (N, N \otimes_A B)$.
Hence, we conclude that  $f$ factors through  $F_1$ that is a finite direct sum of copies of $B$. Thus, $f\in \p$ and therefore, we have the equality $\p=\ker(\Hom_{\K(B)}(N, \omega_N))$.
\end{proof}

The following result is our second main theorem of this section that follows from Corollary~\ref{cor20230118a} and Lemma~\ref{lem20230118b}.

\begin{thm}\label{kernel of omega}
Assume that $\AR$ holds for a DG $B$-module $N$, and let $\p$ be the ideal of $\End _{\K(B)}(N)$ introduced in Lemma~\ref{lem20230118a}. Then, there is an exact sequence
$$
0 \to \p \to \Gamma _{N \otimes _B T} \xra{\omega_N \cdot -} \Gamma _{N \otimes _B T}[1] \to \End _{\K(B)} (N)[1] \to 0
$$
of graded $\Gamma _{N \otimes _B T} $-modules in which $\p$ is regarded as a graded ideal of the ring $\Gamma _{N \otimes _B T}$ concentrated in degree $0$.
\end{thm}

\section{Na\"{\i}ve liftings}\label{sec20230422a}

This section is devoted to the proof of Theorem~\ref{equivalent conditions}, which is our main result in this paper, and its application, namely, Corollary~\ref{cor20230807a}. The proof of this theorem uses our entire work from the previous sections. We start with reminding the reader of the definition of na\"{\i}ve liftability which was first introduced in~\cite{NOY} and further studied in~\cite{NOY1, NOY3, NOY2}.

\begin{para}\label{defn20230127a}
Let $N$ be a semifree DG $B$-module and $N|_A$ denote $N$ regarded as a DG $A$-module via the DG $R$-algebra homomorphism $\varphi\colon A\to B$. We say $N$ is {\it na\"ively liftable to $A$} if
the map  $\pi _N\colon N|_A\otimes_AB\to N$ defined by $\pi_N(n\otimes_Ab)=nb$ is a split DG $B$-module epimorphism. Equivalently, $N$  is na\"ively liftable to $A$ if $\pi_N$ has a right inverse in the abelian category $\C(B)$.
\end{para}

\begin{para}
It is worth mentioning that if $B=A\langle X\rangle$ is a simple free extension of DG algebras, then na\"ively liftability is equivalent to the classical notions of lifting for bounded below semifree DG $B$-modules; see~\cite[Theorem 6.8]{NOY}.
\end{para}

\begin{thm}\label{equivalent conditions}
Assume that $\AR$ holds for a DG $B$-module $N$. Then, the following conditions are equivalent:
\begin{enumerate}[\rm(i)]
\item
$N$ is na\"{\i}vely liftable to $A$;
\item
$\omega_N=0$ as an element of the ring $\Gamma_{N\otimes _B T}$;
\item
$\omega_N$ is nilpotent in $\Gamma_{N\otimes _B T}$, i.e., there exists an integer $n\geq 1$ such that $\omega_N ^n =0$;
\item
The natural inclusion $\End_{\K(B)}(N) \hookrightarrow \Gamma_{N\otimes _B T}$ is an isomorphism;
\item
The ring $\Gamma_{N\otimes _B T}$ is finitely generated as a (right) $\End_{\K(B)}(N)$-module;
\item
$\Gamma_{N\otimes _B T}^i = \Hom _{\K(B)}(N, N \otimes _B T^i)=0$ for all integers $i\geq 1$;
\item
$\Gamma_{N\otimes _B T}^i = \Hom _{\K(B)}(N, N \otimes _B T^i)=0$ for some integer $i\geq 1$;
\item
$\Hom _{\K(B)}(N, N \otimes _B \shift J)=0$;
\item
$N$ is a direct summand of a finite direct sum of copies of $B$ in $\D (B)$.
\end{enumerate}
\end{thm}

The proof of Theorem~\ref{equivalent conditions} is given after the following lemma.

\begin{lem}\label{omega+}
Let $N$  be a semifree DG $B$-module.
Then, the following conditions are equivalent:
\begin{enumerate}[\rm(i)]
\item  $N$  is na\"{\i}vely liftable to $A$;
\item The short exact sequence~\eqref{eq20130118a} is split in $\C^{gr}(T)$;
\item  $\omega_N ^+ =0$  in $\K^{gr}(T)$;
\item  $\omega_N =0$ in $\K^{gr}(T)$.
\end{enumerate}
\end{lem}

\begin{proof}
The implication (i)$\implies$(ii) follows from the fact that $\theta_N=\pi_N\otimes_B\id_T$ in~\eqref{eq20130118a}. The implications (ii) $\Longleftrightarrow$ (iii)$\implies$(iv) are trivial from the definition and Proposition~\ref{triangles}.

For (ii)$\implies$(i) assume that there is a map $\rho\colon N \otimes _B T  \to N \otimes _A T$ in $\C^{gr}(T)$ such that $\theta_N \rho = \id_{N \otimes _BT}$.
Note that $\theta_N$ and $\rho$ are tensor graded and $\pi_N=(\theta_N)^0$. Now we have $\pi_N (\rho)^0 = \id_{N}$, which means that $\pi_N$ splits and $N$ is na\"{\i}vely liftable to $A$.

For (iv)$\implies$(iii) note that $T^+$ is a DG ideal of $T$ and $T/T^+ \cong B$. Hence,  we obtain a short exact sequence
$0 \to N \otimes _B T^+ \to N \otimes _B T \to N \to 0$
in $\C^{gr}(T)$ that yields a commutative diagram
$$
\xymatrix{
\shift^{-1}N[1]  \ar[r] & N \otimes _B T^+ [1] \ar[rr]^{[\varpi_N]} && N \otimes _B T [1] \ar[r] &N[1] \\
&N \otimes _B T \ar[u]_{\omega_N^+}  \ar@{=}[rr]  \ar@{.>}[ul]^{\gamma} && N \otimes _B T \ar[u]^-{\omega_N} & }
$$
in $\K^{gr} (T)$ in which the first row is a triangle. If $\omega_N =0$ in $\K^{gr}(T)$, then $\omega_N^+$ factors through a morphism  $\gamma\colon N \otimes _B T \to \shift^{-1}N$.
On the other hand, we have $\gamma=0$ because $\Hom_{\K^{gr}(T)}(N\otimes_BT,\shift^{-1}N[1])\cong \Hom_{\K(B)}(N,\left(\shift^{-1}N[1]\right)^0)=0$. Therefore, $\omega_N^+=0$, as desired.
\end{proof}


\noindent \emph{Proof of Theorem~\ref{equivalent conditions}.}
The equivalence (i) $\Longleftrightarrow$ (ii) has been proven in Lemma~\ref{omega+}.

(ii) $\Longleftrightarrow$ (iii): The second assertion of Theorem~\ref{omega generates end} assures that $\omega_N = 0$ if and only if $\omega_N ^n=0$ for some integer $n\geq 1$.

The equivalences (ii) $\Longleftrightarrow$ (iv) $\Longleftrightarrow$ (vi) $\Longleftrightarrow$ (vii) follow from Theorem~\ref{omega generates end}.

The equivalence (ii) $\Longleftrightarrow$ (viii) follows from Theorem~\ref{omega generates end} as well because $$\Hom _{\K(B)}(N, N \otimes _B \shift J) = \Gamma _{N\otimes _B T}^1 = \omega_N \cdot \Gamma _{N\otimes _B T}^0.$$

(iv) $\Longleftrightarrow$ (v) is trivial.


For (ix) $\Longleftrightarrow$ (ii), let $\p$ be the ideal of the ring $\End _{\K(B)}(N)$ consisting of all morphisms that factor through a finite direct sum of copies of $B$. Then, we have
\begin{center}
$\omega_N =0 \Longleftrightarrow \p = \End _{\K^{gr}(T)} (N) = \Gamma _{N \otimes _BT} \Longleftrightarrow \id_N\in \p$
\end{center}
where the left equivalence follows from Theorem~\ref{kernel of omega}. Note that $\id_N\in \p$ means that $\id_N$ factors through $B^{\oplus n}$ for some $n\geq 0$, i.e., $N$ is a direct summand of  $B^{\oplus n}$.
\qed

\begin{para}
In Theorem~\ref{equivalent conditions}, if we further assume that $N$ is a semifree resolution of the $\HH_0(B)$-module $\HH_0(N)$, i.e., the natural augmentation map  $N \to \HH_0(N)$ is a quasiisomorphism,
then condition (v) is equivalent to the following:
\vspace{2mm}
\begin{enumerate}[\rm(v')]
\item
{\it The ring $\Gamma_{N\otimes _B T}$ is finitely generated as a (right) $\HH_0(B)$-module.}
\end{enumerate}
\vspace{2mm}
Note that since $N$ is a perfect DG $A$-module, $\HH_0(N)$ is finitely generated over $\HH_0(A)$ and we also have the ring homomorphism $\HH_0(A) \to \HH_0(B)$. Therefore, in this case, $\End _{\K(B)} (N) \cong \End _{\HH_0(B)}(\HH_0(N))$, where $\HH_0(N)$ is finitely generated over $\HH_0(B)$.
\end{para}

As an application of Theorem~\ref{equivalent conditions}, we can show Corollary~\ref{cor20230807a} below which provides an affirmative answer to~\cite[Question 4.10]{NOY2} under the $\AR$ condition. In order to explain this, we need the following preparations. 

\begin{para}
Before our next discussion in~\ref{para20230807s}, we need to make a special arrangement: for an integer $\ell$, a map $\omega_N^{\ell}\colon N\to N\otimes_B(\shift J)^{\otimes_B\ell}$ was defined in~\cite[Remark 4.9]{NOY2}. This map is different from the $\ell$-th power of $\omega_N$ that we define in~\ref{para20230115e} of the present paper as an element in $\Gamma_{N\otimes_BT}$. Hence, in order to avoid confusion made by using the same notation for two different objects, in this paper we use the notation $\chi^{\ell}_N$ instead of the notation $\omega_N^{\ell}$ defined in~\cite[Remark 4.9]{NOY2}.
\end{para}

\begin{para}\label{para20230807s}
Let $(\mathbb{B},\mathbb{D})$ be the semifree resolution of the DG $B^e$-module $B$ constructed in~\cite{NOY2}. Following~\cite[\S 4]{NOY2}, for each integer $\ell\geq 0$ we have that $(\mathbb{B}^{\leq \ell}, \mathbb{D}|_{\mathbb{B}^{\leq \ell}} )$ is a DG $B^e$-submodule of $(\mathbb{B}, \mathbb{D})$, where
$$
\mathbb{B} ^{\leq \ell} := B \otimes _A \left(\bigoplus _{n=0}^{\ell}  (\shift J) ^{\otimes _Bn}\right).
$$
Note that  $\mathbb{B} ^{\leq 0} = (B^e,\,d^{B^e})$. Moreover, let $\omega ^\ell\colon \mathbb{B} \to \mathbb{B}/\mathbb{B}^{\leq \ell}$ be the natural DG $B^e$-module homomorphism for each integer $\ell \geq 0$.
For such integers, it follows from~\eqref{eq20230822a} and~\cite[Remark 4.9]{NOY2} that $\w_N^{\ell}=\chi_N^{\ell}\otimes_B\id_T$. Ignoring isomorphisms in $\K(B)$, we see that $\omega_N^{\ell}=[\chi_N^{\ell}\otimes_B\id_T]=[\id_N\otimes_B\omega^{\ell}\otimes_B\id_T]$.

Let $N$ be a semifree DG $B$-module. By~\cite[Theorem 4.8]{NOY2}, 
the following conditions are equivalent.
\begin{enumerate}[\rm(i)]
\item
$N$ is na\"{\i}vely liftable to $A$;
\item
The DG $B$-module homomorphism  
$\id_N \otimes _B \omega^0$
is null-homotopic;
\item
The DG $B$-module homomorphisms  
$\id_N \otimes _B \omega^\ell$
are null-homotopic for \emph{all} integers $\ell \geq 0$. 
\end{enumerate} 
As we mention in~\cite[Question 4.10]{NOY2}, it is natural to ask whether these conditions are equivalent to the following:
\begin{enumerate}[\rm(iv)]
\item
The DG $B$-module homomorphism $\id_N \otimes _B \omega^\ell$ is null-homotopic for \emph{some} integer $\ell\geq 1$.
\end{enumerate}
As we stated above, the following result provides an affirmative answer to this question under the $\AR$ condition and follows immediately from~\ref{para20230807s} along with the fact that conditions (ii) and (iii) in Theorem~\ref{equivalent conditions} are equivalent.
\end{para}

\begin{cor}\label{cor20230807a}
Assume that $\AR$ holds for a DG $B$-module $N$. Under the setting of~\ref{para20230807s}, conditions (i) through (iv) are equivalent.
\end{cor}

\begin{para}
It is worth highlighting that if $\AR$ and $\Ar$ hold for a DG $B$-module $N$, then by Theorem~\ref{thm20230425a} we automatically have the vanishing of $\Gamma_{N\otimes_BT}$-modules
\begin{gather*}
{}^*\!\Hom _{\K^{gr}(T) } (N \otimes _B T, N \otimes _B \shift^1 T) \\
{}^*\!\Hom _{\K^{gr}(T) } (N \otimes _B T, N \otimes _B \shift^2 T) \\
\vdots
\end{gather*}
However, na\"{\i}ve liftability of $N$ is independent of the above modules and by Theorem~\ref{equivalent conditions}, it is detected only by the vanishing of $\omega_N$ as an element of the $\Gamma_{N\otimes_BT}$-module
$$
\Gamma_{N\otimes_BT}={}^*\!\Hom _{\K^{gr}(T) } (N \otimes _B T, N \otimes _B \shift^0 T).
$$
\end{para}

We conclude this section with the following conjecture that we call \emph{Na\"{\i}ve Lifting Conjecture}; compare with the same named conjecture in~\cite{NOY1}.

\begin{conj}\label{conj20230520a}
If $\AR$ and $\Ar$ hold for a DG $B$-module $N$, then $N$ is na\"{\i}vely liftable to $A$, i.e., one of the equivalent conditions in Theorem~\ref{equivalent conditions} holds.
\end{conj}

\appendix
\section
{{\small \!\!\!The Auslander-Reiten Conjecture and \,\\ conditions $\AR$ and $\Ar$\!}}
\label{sec20230521a}

In this appendix, we explain that conditions $\AR$ and $\Ar$ are derived from translating the assumptions in the \underline{A}uslander-\underline{R}eiten Conjecture (Conjecture~\ref{conj20230122a}) into the DG setting; see~\ref{para20230521a}. This is why we call these conditions ``\texttt{AR}''.

\begin{prop}\label{para20230520a}
Let $N$ be a semifree DG $B$-module with $\HH_i(N)=0$ for all $i\neq 0$. Then, $\Hom_{\K(B)}(N,\shift^{\ell}N)=0$ for all $\ell<0$. 
\end{prop}

\begin{proof}
Using the assumption that $\HH_i(N)=0$ for all $i<0$, we may assume that $N$ has a semifree filtration
$$
0=F_{-1}\overset{\iota_{-1}}\hookrightarrow F_0\overset{\iota_0}\hookrightarrow \cdots\hookrightarrow F_{i-1}\overset{\iota_{i-1}}\hookrightarrow F_i\overset{\iota_{i}}\hookrightarrow\cdots \hookrightarrow N
$$
such that for all $i\geq 0$, $$F_{i-1}\xra{\iota_{i-1}} F_i\to \bigoplus(\shift^{i} B)\to \shift F_{i-1}$$ is a triangle in $\K(B)$. It follows from the isomorphism $\Hom_{\K(B)}(\shift^nB,N)\cong \HH_n(N)$ that $\Hom_{\K(B)}(F_n,\shift^{\ell}N)\cong \Hom_{\K(B)}(F_0,\shift^{\ell}N)\cong \prod\HH_{-\ell}(N)=0$ for all $n\geq -1$ and $\ell\leq -1$. Note that there is a triangle
$$
\bigoplus_{n=-1}^{\infty} F_n\xra{\psi}\bigoplus_{n=-1}^{\infty} F_n\to N\to \shift\left(\bigoplus_{n=-1}^{\infty} F_n\right)
$$
where $\psi$ is defined by $\psi|_{F_n}=\id_{F_n}-\iota_n$. This triangle induces the exact sequence
$$
\Hom_{\K(B)}(\shift(\bigoplus_{n=-1}^{\infty} F_n),\shift^{\ell}N)\to\Hom_{\K(B)}(N,\shift^{\ell}N)\to
\Hom_{\K(B)}(\bigoplus_{n=-1}^{\infty} F_n,\shift^{\ell}N).
$$
Thus, since $\Hom_{\K(B)}(\shift(\bigoplus_{n=-1}^{\infty} F_n),\shift^{\ell}N)\cong \prod_{n=-1}^{\infty}\Hom_{\K(B)}(F_n,\shift^{\ell-1}N)=0$ for all $\ell\leq 0$, we conclude that $\Hom_{\K(B)}(N,\shift^{\ell}N)=0$ for all $\ell\leq -1$.
\end{proof}

\begin{para}
In contrast to Proposition~\ref{para20230520a}, one can construct an example of a semifree DG $B$-module $N$ such that $\HH_i(N)\neq 0$ for some $i\neq 0$ and at the same time $\Hom_{\K(B)}(N,\shift^{-1}N)\neq 0$. To see this, let $a\neq 0$ be an element in the ring $R$ with $a^2=0$, and let $N$ be a semifree DG $R$-module with a semifree basis $\{e_0,e_1\}$, where $|e_0|=0$, $|e_1|=1$, and $\partial^N(e_1)=e_0a$. In other words, $N$ is the Koszul complex $K^R(a)$. Then, $\HH_1(N)\neq 0$ because $a$ is a zero divisor in $R$. Now, considering the DG $R$-module homomorphism $f\colon N\to \shift^{-1} N$ defined by $f(e_0)=e_1a$ and $f(e_1)=0$, we see that $f$ is a non-zero element in $\Hom_{\K(B)}(N,\shift^{-1}N)$. 
\end{para}

The following is proved by the same technique as in the proof of Proposition~\ref{para20230520a}.

\begin{prop}\label{para20230520b}
If $X$ is a DG $B$-module with $\HH_i(X)=0$ for all $i\geq 1$ and $F$ is a semifree DG $B$-module with $\HH_i(F)=0$ for all $i<0$, then $\Hom_{\K(B)}(F,\shift^iX)=0$ for all $i<0$.
\end{prop}

\begin{cor}\label{para20230520b'}
If $\HH_i(B)=0$ for all $i\geq 1$ and $N$ is a semifree DG $B$-module with $\HH_i(N)=0$ for all $i<0$, then $\Hom_{\K(B)}(N,\shift^iB)=0$ for all $i<0$.
\end{cor}

\begin{para}\label{para20230521a}
We work in the setting of Conjecture~\ref{conj20230122a} and discussion~\ref{para20230424s}. Note that the $R$-module $M$ is regarded as a DG $Q'$-module via the natural augmentation $Q' \to R$. Assume that $N\xra{\simeq} M$ is a semifree resolution of the DG $Q'$-module $M$. It follows from Proposition~\ref{para20230520a} and Corollary~\ref{para20230520b'} (with $B=Q'$) that $\Ext^i_{Q'}(N,N\oplus Q')=0$ for all $i\neq 0$ (e.g., $\AR$ and $\Ar$ hold for the DG $Q'$-module $N$) when we translate the Ext vanishing assumptions for $M$ in Conjecture~\ref{conj20230122a}
into the DG setting. With this explanation, if Conjecture~\ref{conj20230520a} holds (i.e., if one of the equivalent conditions in Theorem~\ref{equivalent conditions} holds for $N$ under the Ext vanishing assumptions of Conjecture~\ref{conj20230122a} for $M$), then Conjecture~\ref{conj20230122a} holds.
\end{para}




\providecommand{\bysame}{\leavevmode\hbox to3em{\hrulefill}\thinspace}
\providecommand{\MR}{\relax\ifhmode\unskip\space\fi MR }
\providecommand{\MRhref}[2]{%
  \href{http://www.ams.org/mathscinet-getitem?mr=#1}{#2}
}
\providecommand{\href}[2]{#2}

\end{document}